\documentclass[11pt,a4paper,english]{amsart}
\usepackage[T1]{fontenc}
\usepackage[utf8]{inputenc}
\usepackage{varioref}
\usepackage{nicefrac}
\usepackage{graphicx}
\usepackage{setspace}
\usepackage{amssymb}

\makeatletter
\numberwithin{equation}{section} 
\numberwithin{figure}{section} 
  \@ifundefined{theoremstyle}{\usepackage{amsthm}}{}
  \theoremstyle{plain}
  \newtheorem{thm}{Theorem}[section]
  \theoremstyle{definition}
  \newtheorem{defn}[thm]{Definition}
  \theoremstyle{plain}
  \newtheorem{lem}[thm]{Lemma}
  \theoremstyle{remark}
  \newtheorem{rem}[thm]{Remark}

\usepackage{cancel}
\usepackage[varg]{pxfonts}

\makeatother

\usepackage{babel}

\begin{document}

\title{Test Functions in Constrained Interpolation}

\author{James Pickering}

\begin{abstract}
We give a set of test functions for the interpolation problem on $H_{1}^{\infty}$,
the constrained interpolation problem studied by Davidson, Paulsen,
Raghupathi and Singh. We show that this set of test functions is minimal.
\end{abstract}

\thanks{This paper is based on work contributing to the author's PhD thesis,
at the University of Newcastle-upon-Tyne, under the supervision of
Michael Dritschel. This work is funded by the Engineering and Physical
Sciences Research Council.}

\address{James Pickering\\
Department of Mathematics\\
University of Newcastle-upon-Tyne\\
Newcastle-upon-Tyne\\
NE8 3JT\\
United Kingdom}

\email{james.pickering@ncl.ac.uk}

\urladdr{http://www.jamespic.me.uk}

\maketitle

\section{Background and Introduction}

There has been increased interest in constrained interpolation problems
recently. These problems have many of the unusual characteristics
of harder interpolation problems (they generally require collections
of kernels, and have interesting behavior in the case of matrix valued
interpolation, much like interpolation on multiply-connected domains),
but are simple enough that we can do calculations explicitly (the
kernels are typically rational functions), and we can reuse much of
the theory of interpolation on the unit disc.

The basic problem, as discussed in \cite{H1Infinity}, is the following:
Under what circumstances can we find a bounded, holomorphic function
$f$ on the disc, with zero derivative at $0$ (that is, \[
f\in H_{1}^{\infty}:=\left\{ g\in H^{\infty}:\, g^{\prime}(0)=0\right\} \]
 for our function $f$), which takes prescribed values $w_{1}$, $\ldots$,
$w_{n}$ at prescribed points $z_{1}$, $\ldots$, $z_{n}$. The solution,
as given in \cite{H1Infinity}, is analogous to the Nevanlinna-Pick
theorem; such a function exists if and only if \[
\left((1-w_{i}\overline{w_{j}})k^{s}(z_{j},\, z_{i})\right)_{i,j=1}^{n}\geq0\quad\forall s\in S^{2}\,,\]
 where $S^{2}$ denotes the real 2-sphere, and the kernels $k^{s}$
are a particular class of kernels, parameterised by points $s$ on
the sphere%
\footnote{They actually give two, different, Nevanlinna-Pick type theorems,
but the second is not relevant here.%
}.

In this paper, we look at interpolation with test functions in $H_{1}^{\infty}$.
The test function approach to interpolation originates with Agler
(see \cite{AglerPick} for a treatment of Agler's approach to this
subject, as well as for general background on the field), although
we'll be taking results and notation from Dritschel and McCullough's
paper (\cite{DritschelInterpolation}) on the subject.

We give a set of test functions for $H_{1}^{\infty}$, broadly following
the approach of \cite{AglerComputation}, via a Herglotz representation
for $H_{1}^{\infty}$. These test functions turn out to be rational
functions, and the set of test functions is parameterised by the sphere.
We show, using techniques similar to those in \cite{DritschelInterpolation},
that our set of test functions is minimal. We give some indication
of how these techniques could yield test functions for other types
of constrained interpolation problem, although the theory appears
to be less elegant in these situations.

We'll also introduce the idea of \emph{differentiating kernels}. These
are a simple analogue of reproducing kernels, and whilst they're not
particularly interesting in and of themselves, they have proved to
be a useful tool when working with problems of this sort.

\section{\label{sec:Differentiating-Kernels}Differentiating Kernels}

In this paper, it's convenient to introduce \emph{differentiating
kernels}. These are along much the same lines as reproducing kernels:
We know by Cauchy's integral formula that differentiation is a bounded
linear functional on $H^{2}$, so for each $x\in\mathbb{D}$, and
for each $i=0,\,1,\,2,\,\ldots$, there exists some function $k\left(x^{(i)},\,\cdot\right)\in H^{2}$
such that \[
f^{(i)}(x)=\left\langle f(\cdot),\, k\left(x^{(i)},\,\cdot\right)\right\rangle \,.\]

The above argument also holds for $H^{2}(R)$, for any finitely connected
planar domain $R$. In the case of $H^{2}[=H^{2}(\mathbb{D})]$, we
can use Cauchy's integral formula to calculate $k(x^{(i)},\, y)$
explicitly. We note that complex contour integration is with respect
to $dz=2\pi i.z\, ds$, where $s$ is normalised arc length measure
(the measure on $H^{2}$). So, if we let $f\in H^{\infty}$, then
\begin{align*}
f^{(n)}(x) & =\frac{n!}{2\pi i}\oint_{\mathbb{T}}\frac{f(z)}{(z-x)^{n+1}}dz\\
 & =\frac{n!}{\cancel{2\pi i}}\int_{\mathbb{T}}\frac{\cancel{2\pi i}\, z\, f(z)}{(z-x)^{n+1}}ds\\
 & =\int_{\mathbb{T}}f(z)\overline{\left(n!\frac{\overline{z}}{(\overline{z}-\overline{x})^{n+1}}\right)}ds\\
 & =\int_{\mathbb{T}}f(z)\overline{\left(n!\frac{z^{-1}}{(z^{-1}-\overline{x})^{n+1}}\right)}ds\\
 & =\int_{\mathbb{T}}f(z)\overline{\left(n!\frac{z^{n}}{(1-\overline{x}z)^{n+1}}\right)}ds\\
 & =\left\langle f(z),\, n!\frac{z^{n}}{(1-\overline{x}z)^{n+1}}\right\rangle _{z}\,.\end{align*}
 Since $H^{\infty}$ is dense in $H^{2}$, this also holds for $f\in H^{2}$.
We can now see that \[
k(x^{(n)},\, y)=n!\frac{y^{n}}{(1-\overline{x}y)^{n+1}}=\frac{\partial^{n}}{\partial\overline{x}^{n}}k(x,\, y)\]
 where $k(x,\, y)=(1-\overline{x}y)^{-1}$ is the ordinary Szegő kernel.
For brevity, we write $k_{x^{(i)}}$ for the function $k(x^{(i)},\,\cdot)$.

If $M_{f}$ is the multiplication operator of $f$ on $H^{\infty}$,
these differentiating kernels satisfy \[
M_{f}^{*}\frac{k_{x^{(n)}}}{n!}=\sum_{i=1}^{n}\frac{f^{(i)}(x)^{*}}{i!}\frac{k^{(n-i)}}{(n-i)!}\,,\]
 as \begin{align*}
\left\langle g,\, M_{f}^{*}\,\frac{k_{x^{(n)}}}{n!}\right\rangle  & =\left\langle M_{f}g,\,\frac{k_{x^{(n)}}}{n!}\right\rangle \\
 & =\left\langle fg,\,\frac{k_{x^{(n)}}}{n!}\right\rangle \\
 & =\frac{1}{n!}(fg)^{(n)}(x)\\
 & =\frac{1}{n!}\sum_{i=1}^{n}\left(\begin{smallmatrix}n\\
i\end{smallmatrix}\right)f^{(i)}(x)g^{(n-i)}(x)\\
 & =\sum_{i=1}^{n}\frac{f^{(i)}(x)}{i!}\frac{g^{(n-i)}(x)}{(n-i)!}\\
 & =\left\langle g,\,\sum_{i=1}^{n}\frac{\overline{f^{(i)}(x)}}{i!}\frac{k_{x^{(n-i)}}}{(n-i)!}\right\rangle \,.\end{align*}

\section{Test Functions}

\subsection{Definitions}

We refer the reader to \cite{DritschelInterpolation} for more in
depth discussion of test functions%
\footnote{Many of these ideas are also covered in \cite{AglerPick}, although
our notation is largely taken from \cite{DritschelInterpolation}.%
}. For our purposes, we will need some basic definitions. Note that
despite the similarity in name, these test functions are unrelated
to their namesakes in distribution theory

\begin{defn}
A set $\Psi$ of complex valued functions on a set $X$ is a set of
test functions if:
\begin{enumerate}
\item For each $x\in X$, \[
\sup\left\{ \left|\psi(x)\right|:\,\psi\in\Psi\right\} <1\,.\]

\item If $F$ is a finite set with $n$ elements, then the unital algebra
generated by $\Psi\vert_{F}$ (the restriction of $\Psi$ to $F$)
is $n$-dimensional (that is, $\Psi$ separates points).
\end{enumerate}
\end{defn}

\begin{defn}
For any collection of test functions $\Psi$, we define a set of positive
kernels \[
\mathcal{K}_{\Psi}:=\left\{ k:x\times x\to\mathbb{C}\vert\,(1-\psi(x)\overline{\psi(y)})k(x,\, y)\geq0\,\forall\psi\in\Psi\right\} \,.\]
 From this, we define a normed space $H^{\infty}(\mathcal{K}_{\Psi})$.
We say a function $f:X\to\mathbb{C}$ is in $H^{\infty}(\mathcal{K}_{\Psi})$
with $\left\Vert f\right\Vert _{H(\mathcal{K}_{\Psi})}\leq1$ if \[
(1-f(x)\overline{f(y)})k(x,\, y)\geq0\,\forall k\in\mathcal{K}_{\Psi}\,.\]
 Since we have defined the unit ball of $H^{\infty}(\mathcal{K}_{\Psi})$,
by extension we have defined the whole of $H^{\infty}(\mathcal{K}_{\Psi})$,
and given its norm.
\end{defn}

\subsection{Zero Norm Probability Measures}

We start out by finding a set of test functions for $H_{1}^{\infty}$.
Since test functions have norm 1 or less, the möbius transform\[
m:\, z\to\frac{1+z}{1-z}\]
(which takes the unit disc to the right half plane), takes test functions
to functions with positive real part. Since $H_{1}^{\infty}\subseteq H^{\infty}$,
our test functions must have a Herglotz representation (see Theorem
1.1.19 of \cite{AglerComputation}), so if $\psi$ is a test function,
$f(z)=\frac{1+\psi(z)}{1-\psi(z)}$, and $f(0)>0$,%
\footnote{This condition isn't as problematic as it looks%
} then \[
f(z)=\int_{\mathbb{T}}\frac{w+z}{w-z}d\mu(w)\]
 for some positive measure $\mu$. For our test function to be in
$H_{1}^{\infty}$, we also need that $\psi^{\prime}(0)=0$. We can
see that since \[
f^{\prime}(z)=\psi^{\prime}(z)\, m^{\prime}\left(\psi(z)\right)\]
 and $m^{\prime}\neq0$, we have that $\psi^{\prime}(0)=0$ if and
only if $f^{\prime}(0)=0$.

Now, \[
f^{\prime}(z)=\int_{\mathbb{T}}\left(\frac{d}{dz}\frac{w+z}{w-z}\right)d\mu(w)=\int_{\mathbb{T}}\frac{2w}{(w-z)^{2}}d\mu(w)\]
 so \[
f^{\prime}(0)=\int_{\mathbb{T}}\frac{2}{w}d\mu(w)=2\int_{\mathbb{T}}\overline{w}d\mu(w)=2\overline{\int_{\mathbb{T}}wd\mu(w)}\]
 so if $\mu$ is a probability measure (which it will later be convenient
to assume it is), then the condition that $\psi^{\prime}(0)=0$ is
equivalent to the condition that $\mathbb{E}(\mu)=0$, so $\mu$ has
\emph{zero-mean}. It should be noted here that requiring $\mu$ to
be a probability measure is equivalent to requiring that $f(0)=1$,
or equivalently still, that $\psi(0)=0$.

We have proved the following:

\begin{thm}
\label{thm:ZeroMean}The analytic function $\psi$ has $\left\Vert \psi\right\Vert _{\infty}\leq1$,
$\psi(0)=0$, and $\psi^{\prime}(0)=0$ if and only if the corresponding
measure is a zero-mean probability measure.
\end{thm}

\subsection{Extreme Directions\label{sub:Extreme-Directions}}

We'll be using a lot of techniques and definitions from \cite{AglerComputation}.
In that paper, they used the convention that if $X$ was a {}``real
function space'' in some sense, then $X^{h}$ is the set of all real
functions in $X$ corresponding to holomorphic functions. Here, we'll
use the convention that $X^{1}$ is the set of all real functions
in $X$ corresponding to holomorphic functions with zero derivative
at 0, in ways that should be fairly clear.

If $L_{\mathbb{R}}^{2}(\mathbb{T})$ is defined in the usual way,
then $L_{\mathbb{R}}^{2,1}(\mathbb{T})$ is the set of all functions
in $L_{\mathbb{R}}^{2}(\mathbb{T})$ which are the real part of an
analytic function with $f^{\prime}(0)=0$. It is easy to see that
$\left(L_{\mathbb{R}}^{2,1}(\mathbb{T})\right)^{\perp}=\text{span}\left\{ \text{Im}z,\,\text{Re}z\right\} $,
so $L_{\mathbb{R}}^{2,1}(\mathbb{T})$ has co-dimension 2 in $L_{\mathbb{R}}^{2}(\mathbb{T})$.
If $M_{\mathbb{R}}(\mathbb{T})$ is the space of finite regular real
Borel measures%
\footnote{Remember that we can associate a harmonic function to a measure $\mu\in M_{\mathbb{R}}(\mathbb{T})$
via the Poisson kernel.%
} on $\mathbb{T}$, and $C_{\mathbb{R}}(\mathbb{T})$ is the space
of real continuous functions on $\mathbb{T}$ ($M_{\mathbb{R}}(\mathbb{T})$
is the dual of $C_{\mathbb{R}}(\mathbb{T})$, under the weak-{*} and
uniform topologies, respectively), then as in \cite{AglerComputation}:\[
M_{\mathbb{R}}^{1}(\mathbb{T})=\left\{ \text{Im}z,\,\text{Re}z\right\} ^{\perp}\]
 and \[
C_{\mathbb{R}}^{1}(\mathbb{T})^{\perp}=\text{span}\left\{ \text{Im}z\, ds,\,\text{Re}z\, ds\right\} \,.\]

We also need to make use of extreme directions. We say a vector $x$
in a cone $C$ is an extreme direction in $C$ if, to have $x=x_{1}+x_{2}$
for some $x_{1},\, x_{2}\in C$ we need $x_{1}=tx$ and $x_{2}=sx$
for some $s,\, t\geq0$.

This allows us to formulate the following:

\begin{thm}
\label{thm:ExtremeDirections}Let $E=\left\{ \mu\in M_{\mathbb{R}}^{1}(\mathbb{T}):\,\mu\geq0\right\} $.
If $\mu$ is an extreme direction in $E$, then $\mu$ is supported
at three or fewer points on $\mathbb{T}$.
\end{thm}
\begin{proof}
Suppose $\mu$ is supported on four or more points in $\mathbb{T}$,
and divide the support of $\mathbb{T}$ into four parts, $\Delta_{1}$
to $\Delta_{4}$. Let $\mu_{i}=\chi_{\Delta_{i}}\mu$, where $\chi_{\Delta}$
is the indicator function on $\Delta$. Let $\mathcal{M}=\text{span}\{\mu_{1},\,\ldots,\,\mu_{4}\}$.
The dimension of $\mathcal{M}$ is 4, and since $M_{\mathbb{R}}^{1}(\mathbb{T})$
has co-dimension 2 in $M_{\mathbb{R}}(\mathbb{T})$, we must have
that \[
\text{dim}\left(\mathcal{M}\cap M_{\mathbb{R}}^{1}(\mathbb{T})\right)\geq2\,.\]
Therefore, there must exist a $\nu\in\mathcal{M}\cap M_{\mathbb{R}}^{1}(\mathbb{T})$
which is linearly independent of $\mu$. Since a measure $\alpha_{1}\mu_{1}+\cdots+\alpha_{4}\mu_{4}\in\mathcal{M}$
is positive whenever $\alpha_{1},\,\ldots,\,\alpha_{4}$ are all positive,
we can choose an $\epsilon>0$ small enough that $\mu\pm\epsilon\nu\geq0$.
Therefore, $\frac{1}{2}(\mu\pm\epsilon\nu)\in E$, but \[
\mu=\frac{1}{2}(\mu+\epsilon\nu)+\frac{1}{2}(\mu-\epsilon\nu)\]
 so $\mu$ is not an extreme direction in $E$.
\end{proof}
We can combine this with Theorem \vref{thm:ZeroMean}, that $\mathbb{E}(\mu)=0$
for $\mu\in E$.

If $\mu$ is supported at one point of $\mathbb{T}$, then it is clearly
impossible to have $\mathbb{E}(\mu)=0$.

If $\mu$ is supported at two points, then $0$ must be in the convex
hull of these two points (the line between them), and so the two points
must lie opposite each other on the circle, and both have equal weight
(if $\mu$ is to be a probability measure, this weight must be $\frac{1}{2}$).

If $\mu$ is supported at three points, then these points must be
such that $0$ is in the \emph{interior} of their convex hull (that
is, the interior of the triangle they form; if $0$ lies on one of
the lines of the triangle, then $\mu$ will only be supported on the
two points at either end of the line, so this is the degenerate case
we had before). We can also see that if $\mu$ is a probability measure,
then its weights are uniquely determined by the points it supports
-- the weights are precisely the barycentric co-ordinates of $0$,
with respect to the three points of the triangle.

We can now note that, if a measure $\mu$ in $E$ is supported on
three or fewer points, it is uniquely determined by those points,
up to multiplication by a scalar. In particular, if $\mu=t_{1}\mu_{1}+t_{2}\mu_{2}$,
for some $\mu_{1},\,\mu_{2}\in E$, then $\mu_{1}$ and $\mu_{2}$
must be supported on a subset of the support of $\mu$, so must be
scalar multiples of $\mu$, so we have characterised the extreme directions
in $E$.

Note that we can rescale any non-zero measure in $E$ to a probability
measure.

\subsection{Some Topology}

This is a convenient time to talk about the {}``space'' of test
functions. %
\begin{figure}
\includegraphics[width=7cm]{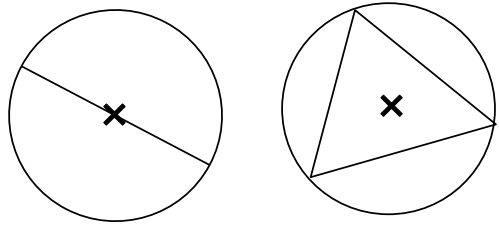}\caption{\label{fig:Topology}Types of Element in $\widehat{\Theta}$}

\end{figure}
We define a set $\widehat{\Theta}$, containing two types of element,
as shown in Figure \vref{fig:Topology}:

\begin{itemize}
\item diameters of the circle
\item triangles, with vertices on the circumference of the circle, and the
centre of the circle in their interior
\end{itemize}
To topologise this set, we say that a sequence of triangles converges
to a triangle if its points converge, and if we have a sequence of
triangles where the centre of the circle seems to converge to one
of the lines, then we say the sequence converges to the diameter,
as in Figure \vref{fig:Convergence}.%
\begin{figure}
\includegraphics[width=7cm]{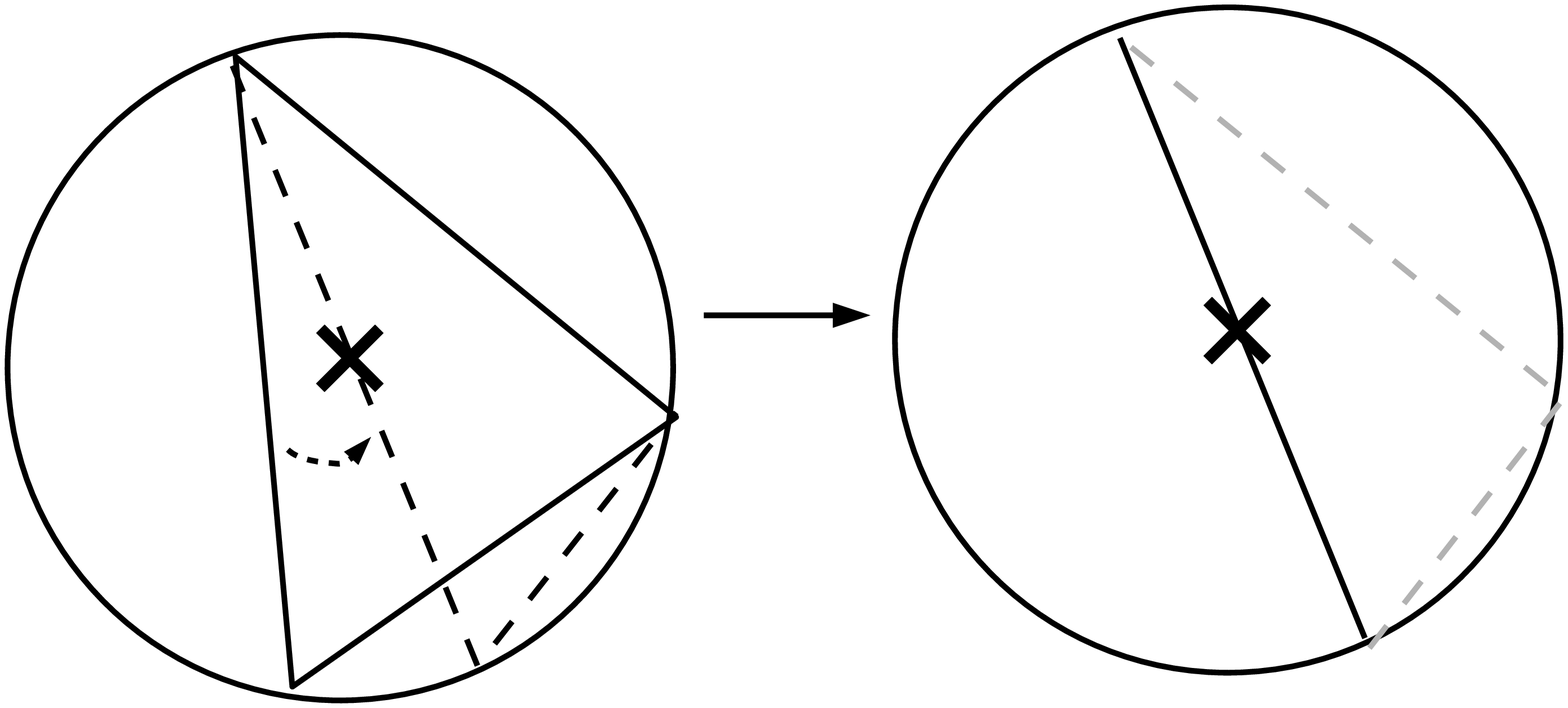}\caption{\label{fig:Convergence}Convergence in $\widehat{\Theta}$}

\end{figure}
 Essentially, the third point on the triangle disappears, which makes
sense, as in the case of the zero-mean probability measures above,
the weight of the third point would tend to zero -- in fact, if we
identify points in $\widehat{\Theta}$ with zero-mean probability
measures on the circle, as in the discussion following Theorem \ref{thm:ExtremeDirections},
we can see that this topology corresponds exactly to the weak-{*}
topology on $M_{\mathbb{R}}(\mathbb{T})$.

When we're dealing with test functions (we haven't defined what these
are yet, but this is the most convenient place for this discussion;
it may make sense to re-read this paragraph later), we can safely
identify two test functions if one is a constant, unimodular multiple
of the other. If $\mu$ is a zero-mean probability measure supported
on $n$ points, then the test function $\psi_{\mu}$ induced by $\mu$
will be an $n$-to-one Blaschke product, with $\psi_{\mu}^{\prime}(0)=0$,
$\psi_{\mu}(0)=0$ and $\psi_{\mu}(w)=1$ precisely when $w$ is in
the support of $\mu$. Conversely, if we have such a function $\psi$,
then there is a corresponding zero-mean probability measure $\mu_{\psi}$,
with support $\psi^{-1}(\{1\})$. If we have a test function $\psi_{\mu}$,
corresponding to a $\mu\in\widehat{\Theta}$, then $\widetilde{\psi}:=\overline{\psi_{\mu}(-1)}\psi_{\mu}$
is a two-or-three-to-one inner function of the type required, and
since $\widetilde{\psi}(-1)=1$, the corresponding measure $\widetilde{\mu}$
is supported at $-1$.

If we define an equivalence relation $\sim$ on $\widehat{\Theta}$
by \[
\mu_{1}\sim\mu_{2}\iff\psi_{\mu_{1}}=\lambda\psi_{\mu_{2}}\text{ for some }\left\Vert \lambda\right\Vert =1\]
 then we can define $\Theta:=\widehat{\Theta}/\sim$. By the above
reasoning, $\Theta$ can be thought of as the set of all triangles
or diameters in $\widehat{\Theta}$ with a vertex at $-1$ (i.e, on
the leftmost point of the circle). When viewed in this sense, there
is only one diameter in $\Theta$.

It's interesting to note that $\Theta$ is homeomorphic to $S^{2}$.
We show (using a certain amount of hand-waving) that $\Theta$ is
homeomorphic to $\mathbb{C}\cup\{\infty\}$. First, we set the diameter
(in $\Theta$) as $\infty$. This leaves the set of all triangles
in $\Theta$. We know that any triangle in $\Theta$ can be represented
by a triangle with a point at $-1$ (the left hand side of the circle),
so we put point one on the left of the circle, allow point two to
vary over the whole top of the circle (corresponding to the real axis),
and allow point three to vary over the range of points opposite the
arc between point one and point two%
\begin{figure}
\includegraphics[width=4cm]{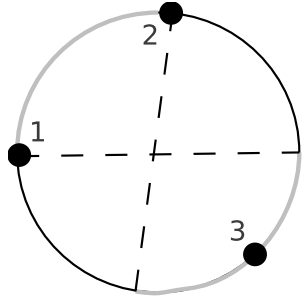}

\caption{\label{fig:Theta}The setup for $\Theta$}

\end{figure}
 (corresponding to the imaginary axis) as in Figure \vref{fig:Theta}.
As either of these points tend towards the edges of their ranges,
the triangle tends towards the diameter -- which corresponds to $\infty$.

\subsection{The Test Functions}

If we take the set of probability measures in $M_{\mathbb{R}}^{1}(\mathbb{T})$,
this is a subset of $E$, and in fact corresponds to $E_{\rho}$,
in the sense of Lemma 3.4 of \cite{AglerComputation}, where $\rho(\mu)=\mu(\mathbb{T})$.
We know, therefore, that $E_{\rho}$ is convex, and we showed before
that its extreme points are given by $\widehat{\Theta}$. $E_{\rho}$
is compact by the Banach-Alaoglu theorem, so we can apply the Krein-Milman
theorem, and we see that for any $\mu\subseteq E_{\rho}$, there exists
a probability measure $\nu_{\mu}$ on $\widehat{\Theta}$ such that%
\footnote{We are, confusingly but unavoidably, talking about integrating a measure-valued
function (that's $\vartheta$), with respect to a measure (that's
$\nu_{\mu}$). Also, the points in the space that $\nu_{\mu}$ integrates
over (that's $\widehat{\Theta}$) are, themselves, measures (they're
measures on $\mathbb{T}$)%
} \[
\mu=\int_{\widehat{\Theta}}\vartheta d\nu_{\mu}(\vartheta)\]

Now, probability measures in $M_{\mathbb{R}}^{1}(\mathbb{T})$ correspond
to analytic functions $f$ on $\mathbb{D}$ with positive real part,
$f^{\prime}(0)=0$, and $f(0)=1$, using the Herglotz representation
theorem. We define \[
h_{\vartheta}(z):=\int_{\mathbb{T}}\frac{w+z}{w-z}d\vartheta(w)\]
 and \[
\psi_{\vartheta}=\frac{h_{\vartheta}-1}{h_{\vartheta}+1}\]
for all $\vartheta\in\widehat{\Theta}$. Using the reasoning above,
we can write \begin{align*}
f(z)= & \int_{\mathbb{T}}\frac{w+z}{w-z}d\mu(w)\\
= & \int_{\mathbb{T}}\frac{w+z}{w-z}\left[\int_{\widehat{\Theta}}\left(d\vartheta(w)\right)d\nu_{\mu}(\vartheta)\right]\\
= & \int_{\widehat{\Theta}}\left[\int_{\mathbb{T}}\frac{w+z}{w-z}d\vartheta(w)\right]d\nu_{\mu}(\vartheta)\\
= & \int_{\widehat{\Theta}}h_{\vartheta}(z)d\nu_{\mu}(\vartheta)\,.\end{align*}

This gives us a new {}``Herglotz representation'', which I'll call
a \emph{Herglotz-Agler representation}:

\begin{thm}
\label{thm:SuperHerglotz}If $f$ is an analytic function on $\mathbb{D}$
with positive real part, $f^{\prime}(0)=0$, and $f(0)>1$, then there
exists some positive real measure $\nu$ on $\widehat{\Theta}$ such
that \[
f(z)=\int_{\widehat{\Theta}}h_{\vartheta}(z)d\nu(\vartheta)\,.\]

\end{thm}
We now prove the main result, which uses terminology from \cite{DritschelInterpolation}.
I'll be using $\Psi$ to refer to the set of test functions associated
with $\Theta$, so $\Psi:=\left\{ \psi_{\vartheta}:\,\vartheta\in\Theta\right\} $):

\begin{thm}
The two spaces $H^{\infty}\left(\mathcal{K}_{\Psi}\right)$ and $H_{1}^{\infty}(\mathbb{D})$
are isometrically isomorphic, that is, $H^{\infty}\left(\mathcal{K}_{\Psi}\right)=H_{1}^{\infty}(\mathbb{D})$
and $\Vert\cdot\Vert_{\mathcal{K}_{\Psi}}=\Vert\cdot\Vert_{H_{1}^{\infty}(\mathbb{D})}$
\end{thm}
\begin{proof}
One way is simple. Since $\Psi\subseteq H_{1}^{\infty}(\mathbb{D})$,
we know that $\mathcal{K}_{\Psi}$ contains the set $\mathcal{K}_{1}^{\infty}$
of reproducing kernels given in \cite{H1Infinity}, so if $\zeta\in H^{\infty}\left(\mathcal{K}_{\Psi}\right)$,
and $\left\Vert \zeta\right\Vert _{\mathcal{K}_{\Psi}}\leq1$, then
\[
\left(\left[1-\zeta(z)\overline{\zeta(w)}\right]k(x,\, w)\right)\geq0\]
 for all $k\in\mathcal{K}_{1}^{\infty}$. Therefore, $\zeta$ must
be in $H_{1}^{\infty}(\mathbb{D})$, with $\left\Vert \zeta\right\Vert _{H_{1}^{\infty}(\mathbb{D})}\leq1$,
so $H^{\infty}\left(\mathcal{K}_{\Psi}\right)\subseteq H_{1}^{\infty}(\mathbb{D})$.

Now, suppose that $\zeta\in H_{1}^{\infty}(\mathbb{D})$ and $\left\Vert \zeta\right\Vert _{H_{1}^{\infty}(\mathbb{D})}\leq1$.
For now, we also suppose that $\zeta(0)=0$. We let\[
f:=\frac{1+\zeta}{1-\zeta}\]
 so \[
\zeta=\frac{f-1}{f+1}\]
 Hence\begin{equation}
1-\zeta(z)\overline{\zeta(w)}=2\frac{f(z)+\overline{f(w)}}{(f(z)+1)(\overline{f(w)}+1)}\label{eq:zetaDef}\end{equation}
 We know that $f$ has positive real part, $f(0)=1$ and $f^{\prime}(0)=0$,
so we can use our Herglotz-Agler representation, Theorem \vref{thm:SuperHerglotz},
and find that there is a measure $\nu$ on $\widehat{\Theta}$ such
that \[
f=\int_{\widehat{\Theta}}h_{\vartheta}d\nu(\vartheta)\,.\]

We can then show, using the definition of $\psi_{\vartheta}$ and
\eqref{eq:zetaDef}, that \[
1-\zeta(z)\overline{\zeta(w)}=\int_{\widehat{\Theta}}\frac{1-\psi_{\vartheta}(z)\overline{\psi_{\vartheta}(w)}}{\left(f(z)+1\right)\left(1-\psi_{\vartheta}(z)\right)\left(1-\overline{\psi_{\vartheta}(w)}\right)\left(\overline{f(w)}+1\right)}d\nu(\vartheta)\]
We know that if $\vartheta\in\widehat{\Theta}$, then $\overline{\psi_{\vartheta}(-1)}\psi_{\vartheta}\in\Psi$.
If we define a positive kernel $\Gamma:\mathbb{D}\times\mathbb{D}\to C_{b}(\Psi)^{*}$
by \[
\Gamma(z,\, w)\alpha=\int_{\widehat{\Theta}}\frac{\alpha\left(\overline{\psi_{\vartheta}(-1)}\psi_{\vartheta}\right)}{\left(f(z)+1\right)\left(1-\psi_{\vartheta}(z)\right)\left(1-\overline{\psi_{\vartheta}(w)}\right)\left(\overline{f(w)}+1\right)}d\nu(\vartheta)\]
 We can then see that\[
1-\zeta(z)\overline{\zeta(w)}=\Gamma(z,\, w)(1-E(z)E(w)^{*})\]
 so $\zeta\in H^{\infty}(\mathcal{K}_{\Psi})$ and $\left\Vert \zeta\right\Vert _{\mathcal{K}_{\Theta}}\leq1$.

To see that this holds when $\zeta(0)\neq0$, simply note that \[
1-\left(\frac{\zeta(z)-a}{1-\bar{a}\zeta(z)}\right)\overline{\left(\frac{\zeta(w)-a}{1-\bar{a}\zeta(w)}\right)}=\frac{\left(1-a\bar{a}\right)\left(1-\zeta(z)\overline{\zeta(w)}\right)}{\left(1-\bar{a}\zeta(z)\right)\left(1-a\overline{\zeta(w)}\right)}\,.\]

Therefore, $\zeta\in H^{\infty}\left(\mathcal{K}_{\Psi}\right)$,
and $\left\Vert \zeta\right\Vert _{\mathcal{K}_{\Psi}}\leq1$, so
$H^{\infty}\left(\mathcal{K}_{\Psi}\right)=H_{1}^{\infty}(\mathbb{D})$
and $\Vert\cdot\Vert_{\mathcal{K}_{\Psi}}=\Vert\cdot\Vert_{H_{1}^{\infty}(\mathbb{D})}$,
as required.
\end{proof}

\subsection{Minimality}

As in \cite{DritschelInterpolation}, we show that this set of test
functions is minimal, in the sense that there is no closed subset
$C$ of $\Psi$ so that $C$ is a set of test functions for $H_{1}^{\infty}$.
First, we need a lemma.

\begin{lem}
\label{lem:MeasureLemma}If, for some measure $\mu$ on a space $C$,
and some separable Hilbert space $H$, $M\in B(H)$, $f\in B(H)\otimes L^{1}(\mu)$,
$f(x)\geq0$ $\mu$-almost everywhere, and $M\geq\int_{C}f(x)\, d\mu(x)$,
then for all scalars $\delta>0$ there exists some $C^{\prime}\subseteq C$
and some scalar $N_{\delta}>0$, such that $\mu(C-C^{\prime})<\delta$
and $M\geq N_{\delta}f(x)$.
\end{lem}
\begin{proof}
Define, for $N>0$, and $\varphi\in H$, \begin{gather*}
C_{N}^{\varphi}=\left\{ x\in C:\,\left\langle \left(M-Nf(x)\right)\varphi,\,\varphi\right\rangle \geq0\right\} \\
C_{N}=\left\{ x\in C:\, M\geq Nf(x)\right\} \\
C_{0}=\bigcup_{N>0}C_{N}\qquad C_{0}^{\varphi}=\bigcup_{N>0}C_{N}^{\varphi}\end{gather*}

We can see that, for any given $\varphi\in H$, $C-C_{0}^{\varphi}$
is a $\mu$-null set. To see this, note that for $x$ to be in $C-C_{0}^{\varphi}$,
we'd need to have $\left\langle M\varphi,\,\varphi\right\rangle =0$
but $\left\langle f(x)\varphi,\,\varphi\right\rangle >0$, however,
\[
\left\langle M\varphi,\varphi\right\rangle \geq\left\langle \int_{C}f(x)\, d\mu(x)\,\varphi,\varphi\right\rangle \geq\int_{C-C_{0}^{\varphi}}\left\langle f(x)\,\varphi,\,\varphi\right\rangle d\mu(x)\geq0\,,\]
which is a contradiction.

If $\Phi$ is a countable dense subset of the unit ball in $H$, we
can see that \[
C_{N}=\bigcap_{\varphi\in\Phi}C_{N}^{\varphi}\]
 and so \begin{multline*}
C-C_{0}=\bigcap_{N>0}\left(C-C_{N}\right)=\bigcap_{N>0}\left(\bigcup_{\varphi\in\Phi}\left[C-C_{N}^{\varphi}\right]\right)\\
=\bigcup_{\varphi\in\Phi}\left(\bigcap_{N>0}\left[C-C_{N}^{\varphi}\right]\right)=\bigcup_{\varphi\in\Phi}\left(C-C_{0}^{\varphi}\right)\,,\end{multline*}
 which is a countable union of null-sets, and so is a null set.

We can now see that as $N\to0$, $\mu(C-C_{N})\to0$, and $M\geq Nf(x)$
for all $x\in C_{N}$, so our result is proved.
\end{proof}
\begin{thm}
\label{thm:Minimal}No proper closed subset $C$ of $\Psi$ is a set
of test functions for $H_{1}^{\infty}$.
\end{thm}
\begin{proof}
Suppose, towards an eventual contradiction, that $C$ is a proper
closed subset of $\Psi$ and $\psi_{0}=\psi_{\vartheta_{0}}\notin C$.
Since $C$ is closed, its complement is open, so we can safely assume
that $\vartheta_{0}$ is not a diameter, and not an equilateral triangle%
\footnote{If we do the calculations, we discover that these two possibilities
correspond to the test functions $z^{2}$ and $-z^{3}$, which are
inconvenient corner cases%
}.

We notice that the differentiating kernels we defined in Section \vref{sec:Differentiating-Kernels}
are rational functions, so we can extend them to the entire Riemann
sphere%
\footnote{The fact that this is homeomorphic to $\Theta$ is not relevant here,
and appears to be a coincidence.%
}, $\mathbb{C}\cup\{\infty\}$. If we do this, then we see that if
$x\neq0$, $k_{x^{(n)}}$ has $n+1$ poles, all at $\overline{x^{-1}}$,
and $k_{0^{(n)}}$ has $n$ poles, all at $\infty$.

The kernels \[
\Delta_{\vartheta}(z,\, w):=\left(1-\psi_{\vartheta}(z)\psi_{\vartheta}(w)^{*}\right)k(w,\, z)\]
 are positive and have rank at most three ($k$ is just the Szegő
kernel). To see this, first note that $\psi_{\vartheta}$ has at most
three zeroes%
\footnote{If $\vartheta$ is a diameter, then it has two zeroes, otherwise it
has three.%
}, and that at least two of them must be at zero, as $\psi_{\vartheta}(0)=0$
and $\psi_{\vartheta}^{\prime}(0)=0$. Also note that $M_{\vartheta}$,
the operator of multiplication by $\psi_{\vartheta}$, is an isometry
on $H^{2}$, so $1-M_{\vartheta}M_{\vartheta}^{*}$ is the projection
onto \[
\mathfrak{M}_{\vartheta}:=\ker M_{\vartheta}^{*}=\text{Span}\left\{ k_{0},\, k_{0^{\prime}},\, k_{a_{\vartheta}}\right\} \]
 where $a_{\vartheta}$ is the third zero of $\vartheta$ (if $\vartheta$
has three zeroes at zero, then $a_{\vartheta}=0^{\prime\prime}$;
if $\vartheta$ is a diameter, then the span doesn't include $k_{a_{\vartheta}}$).

Now, \[
\Delta_{\vartheta}(z,\, w)=\left\langle \left(1-M_{\vartheta}M_{\vartheta}^{*}\right)k_{w},\, k_{z}\right\rangle :=\left\langle P_{\vartheta}k_{w},\, k_{z}\right\rangle =\left\langle P_{\vartheta}k_{w},\, P_{\vartheta}k_{z}\right\rangle \,,\]
 and we can think of this as being a holomorphic function in $z$,
and an antiholomorphic function in $w$. If we think of the antiholomorphic
function as being in the dual of $H^{2}$, then \[
\Delta_{\vartheta}\in H^{2}\otimes\left(H^{2}\right)^{*}\cong B(H^{2})\,.\]

More explicitly, $\Delta_{\vartheta}$ defines an operator on $H^{2}$
as \[
\Delta_{\vartheta}f(z):=\int_{\mathbb{T}}\Delta_{\vartheta}(z,\, w)\, f(w)\, ds(w)\,,\]
 so\begin{align*}
\left\langle \Delta_{\vartheta}f,g\right\rangle  & =\int_{\mathbb{T}}\int_{\mathbb{T}}\overline{g(z)}\,\Delta_{\vartheta}(z,\, w)\, f(w)\, ds(w)\, ds(z)\\
 & =\int_{\mathbb{T}}\int_{\mathbb{T}}\overline{g(z)}\,\left\langle P_{\vartheta}k_{w},\, P_{\vartheta}k_{z}\right\rangle \, f(w)\, ds(w)\, ds(z)\\
 & =\int_{\mathbb{T}}\int_{\mathbb{T}}\left\langle f(w)\, P_{\vartheta}k_{w},\, g(z)\, P_{\vartheta}k_{z}\right\rangle \, ds(w)\, ds(z)\\
 & =\left\langle \int_{\mathbb{T}}f(w)\, P_{\vartheta}k_{w}\, ds(w),\,\int_{\mathbb{T}}g(z)\, P_{\vartheta}k_{z}\, ds(z)\right\rangle \\
 & =\left\langle \int_{\mathbb{T}}f(w)\, P_{\vartheta}k_{w}\, ds(w),\,\int_{\mathbb{T}}g(z)\, P_{\vartheta}k_{z}\, ds(z)\right\rangle _{\mathfrak{M}_{\vartheta}}\end{align*}
 so we have factorised $\Delta_{\vartheta}$ as $A_{\vartheta}^{*}A_{\vartheta}$,
where $A_{\vartheta}:H^{2}\to\mathfrak{M}_{\vartheta}$ is given by
\[
A_{\vartheta}f:=\int_{\mathbb{T}}f(w)\, P_{\vartheta}k_{w}\, ds(w)\,.\]

We also note now, for use later, that $A_{\vartheta}^{*}=I_{\mathfrak{M}_{\vartheta}}$,
the embedding map of $\mathfrak{M}_{\vartheta}$ into $H^{2}$ , as
\begin{align*}
\left\langle A_{\vartheta}f,\, g\right\rangle  & =\left\langle \int_{\mathbb{T}}f(w)\, P_{\vartheta}k_{w}\, ds(w),\, g\right\rangle _{\mathfrak{M}_{\vartheta}}\\
 & =\int_{\mathbb{T}}f(w)\,\left\langle P_{\vartheta}k_{w},\, g\right\rangle _{\mathfrak{M}_{\vartheta}}\, ds(w)\\
 & =\int_{\mathbb{T}}f(w)\,\left\langle P_{\vartheta}k_{w},\, I_{\mathfrak{M}_{\vartheta}}g\right\rangle \, ds(w)\\
 & =\int_{\mathbb{T}}f(w)\,\left\langle k_{w},\, I_{\mathfrak{M}_{\vartheta}}g\right\rangle \, ds(w)\\
 & =\int_{\mathbb{T}}f(w)\overline{\left(I_{\mathfrak{M}_{\vartheta}}g\right)(w)}ds\\
 & =\left\langle f,\, I_{\mathfrak{M}_{\vartheta}}g\right\rangle \,.\end{align*}

We choose any set of four points $F=\{z_{1},\, z_{2},\, z_{3},\, z_{4}\}\in\mathbb{D}$,
and consider the classical Nevanlinna-Pick problem, of finding a contractive
function $\varphi\in H^{\infty}$ such that $\varphi(z_{i})=\psi_{0}(z_{i})$
for all $i$. Since $\Delta_{0}(z,\, w)$ has rank at most three,
the $4\times4$ matrix \[
\left(\left[1-\psi_{0}(z_{i})\psi_{0}(z_{j})^{*}\right]k(z_{j},\, z_{i})\right)_{i,j=1}^{4}\]
 must be singular, so the problem has a unique solution, $\varphi=\psi_{0}$.

Now, if we assume that $C$ is a set of test functions for $H_{1}^{\infty}$,
then by Theorem 2.3 of \cite{DritschelInterpolation} there must be
a positive kernel $\Gamma:F\times F\to C(C)^{*}$ such that \[
1-\psi_{0}(z_{i})\psi_{0}(z_{j})^{*}=\Gamma(z_{i},\, z_{j})\left(1-E(z_{i})E(z_{j})^{*}\right)\,.\]
 Indeed, by Theorem 2.2 of \cite{DritschelInterpolation}, this kernel
must extend to the whole of $\mathbb{D}\times\mathbb{D}$. We can
rewrite this, in our case, by saying that there exists a measure $\mu$
on $C$, and functions $h_{l}(z,\,\cdot)\in L^{2}(\mu)$, for $l=1,\,\ldots,4$,
such that \begin{equation}
1-\psi_{0}(z)\psi_{0}(w)^{*}=\int_{C}\sum_{l=1}^{4}h_{l}(z,\,\psi)h_{l}(w,\,\psi)^{*}\left(1-\psi(z)\psi(w)^{*}\right)d\mu(\psi)\,.\label{eq:real}\end{equation}

Multiplying this equation by $k(z,\, w)$ gives \[
\Delta_{0}(z,\, w)=\int_{C}\sum_{l=1}^{4}h_{l}(z,\,\psi)\Delta_{\psi}(z,\, w)h_{l}(w,\,\psi)^{*}d\mu(\psi)\,.\]

Since $\Delta_{\psi}$ is a positive kernel and a positive operator,
when seen as an operator on $H^{2}$, as above, we can say that for
all $l$, \[
\Delta_{0}(z,\, w)\geq\int_{C}h_{l}(z,\,\psi)\Delta_{\psi}(z,\, w)h_{l}(w,\,\psi)^{*}d\mu(\psi)\,.\]
 We now know, by Lemma \vref{lem:MeasureLemma}, that for any $\delta>0$,
there is a set $C^{\prime}$, and a constant $c_{\delta}>0$ such
that $\mu(C-C^{\prime})<\delta$, and \[
\Delta_{0}(z,\, w)\geq c_{\delta}h_{l}(z,\,\psi)\Delta_{\psi}(z,\, w)h_{l}(w,\,\psi)^{*}\]
 for all $\psi\in C^{\prime}$.

If we use our factorisation of $\Delta_{\vartheta}$ from above, we
see that \[
A_{0}^{*}A_{0}\geq c_{\delta}h_{l}(z,\,\psi)A_{\psi}^{*}A_{\psi}h_{l}(w,\,\psi)^{*}\]
 and so we can apply Douglas' Lemma (\cite{DouglasLemma}), to see
that the range of $h_{l}(\cdot,\,\psi)A_{\psi}^{*}$ is contained
in the range of $A_{0}^{*}$, therefore, there exist constants $c_{1},\,\ldots,\, c_{9}$
so that \begin{align}
h_{l}(\cdot,\,\psi)k_{0} & =c_{1}k_{0}+c_{2}k_{0^{\prime}}+c_{3}k_{a_{0}}\label{eq:sim-1}\\
h_{l}(\cdot,\,\psi)k_{0^{\prime}} & =c_{4}k_{0}+c_{5}k_{0^{\prime}}+c_{6}k_{a_{0}}\label{eq:sim-2}\\
h_{l}(\cdot,\,\psi)k_{a_{\psi}} & =c_{7}k_{0}+c_{8}k_{0^{\prime}}+c_{9}k_{a_{0}}\,.\label{eq:sim-3}\end{align}

By letting $\delta$ go to zero, we see that these equations must
hold for $\mu$-almost-all $\psi\in\Psi$.

Equation \eqref{eq:sim-1} tells us that $h_{l}(\cdot,\,\psi)=c_{1}k_{0}+c_{2}k_{0^{\prime}}+c_{3}k_{a_{0}}$,
as $k_{0}$ is constant We can also see that $h_{l}(\cdot,\,\psi)$
must extend meromorphically to the Riemann sphere, as these kernels
do so. Equation \eqref{eq:sim-2} tells us that $c_{2}=0$, as otherwise
the left hand side of the equation has a triple pole at $\infty$,
but the right hand side has at most only a double pole.

We consider equation \eqref{eq:sim-3} in three cases. Firstly, if
$\psi$ has only two zeroes, there is no equation \eqref{eq:sim-3},
so $h_{l}(\cdot,\,\psi)=c_{1}+c_{3}k_{a_{0}}$. 

If $a_{\psi}\neq0$ , then\begin{gather*}
h_{l}(y,\,\psi)k_{a_{\psi}}(y)=c_{7}k_{0}(y)+c_{8}k_{0^{\prime}}(y)+c_{9}k_{a_{0}}(y)\\
\left[c_{1}+c_{3}\frac{1}{1-\overline{a_{0}}y}\right]\frac{1}{1-\overline{a_{\psi}}y}=c_{7}+c_{8}y+c_{9}\frac{1}{1-\overline{a_{0}}y}\\
c_{1}+c_{3}\frac{1}{1-\overline{a_{0}}y}=c_{7}-c_{7}\overline{a_{\psi}}y+c_{8}y-c_{8}\overline{a_{\psi}}y^{2}+\frac{c_{9}-c_{9}\overline{a_{\psi}}y}{1-\overline{a_{0}}y}\\
c_{1}-c_{1}\overline{a_{0}}y+c_{3}=\left\{ \begin{matrix}c_{7}-c_{7}\overline{a_{\psi}}y-c_{7}\overline{a_{0}}y+c_{7}\overline{a_{\psi}}\overline{a_{0}}y^{2}\\
+c_{8}y-c_{8}\overline{a_{\psi}}y^{2}-c_{8}\overline{a_{0}}y^{2}+c_{8}\overline{a_{\psi}}\overline{a_{0}}y^{3}\\
+c_{9}-c_{9}\overline{a_{\psi}}y\end{matrix}\right.\,.\end{gather*}
Looking at the $y^{3}$ coefficient tells us that $c_{8}=0$, looking
at the $y^{2}$ coefficient then tells us that $c_{7}=0$, and so
comparing the constant and $y$ coefficients, we see that \begin{gather*}
\left\{ \begin{array}{rl}
c_{1}+c_{3} & =c_{9}\\
c_{1}\overline{a_{0}} & =c_{9}\overline{a_{\psi}}\end{array}\right.\\
\left\{ \begin{array}{rl}
c_{1} & =c_{9}\frac{\overline{a_{\psi}}}{\overline{a_{0}}}\\
c_{3} & =c_{9}\left(1-\frac{\overline{a_{\psi}}}{\overline{a_{0}}}\right)\end{array}\right.\end{gather*}
 and so \[
h_{l}(y,\,\psi)=c_{9}\left(\frac{\overline{a_{\psi}}}{\overline{a_{0}}}+\left(1-\frac{\overline{a_{\psi}}}{\overline{a_{0}}}\right)\frac{1}{1-\overline{a_{0}y}}\right)=c\frac{1-\overline{a_{\psi}}y}{1-\overline{a_{0}}y}\]

Alternately, if $a_{\psi}=0$, then equation \eqref{eq:sim-3} becomes
\begin{gather*}
h_{l}(y,\,\psi)y^{2}=c_{7}+c_{8}y+c_{9}\frac{1}{1-\overline{a_{0}}y}\\
c_{1}y^{2}+c_{3}\frac{y^{2}}{1-\overline{a_{0}}y}=c_{7}+c_{8}y+c_{9}\frac{1}{1-\overline{a_{0}}y}\\
c_{1}y^{2}-c_{1}\overline{a_{0}}y^{3}+c_{3}y^{2}=c_{7}+c_{7}\overline{a_{0}}y+c_{8}y+c_{8}\overline{a_{0}}y^{2}+c_{9}\end{gather*}
 Looking at the $y^{3}$ coefficient tells us that $c_{1}\overline{a_{0}}=0$,
so $c_{1}=0$. The $y^{2}$ terms then tell us that $c_{3}=c_{8}\overline{a_{0}}$.
We then see that \[
h_{l}(y,\,\psi)=c_{3}\frac{1}{1-\overline{a_{0}}y}=c\frac{1-\overline{a_{\psi}}y}{1-\overline{a_{0}}y}\]
 as before.

Combining these consequences of equations \eqref{eq:sim-1}-\eqref{eq:sim-3}
with \eqref{eq:real} gives a more explicit realisation than the one
in \eqref{eq:real}; that is, \begin{multline*}
1-\psi_{0}(z)\psi_{0}(w)^{*}=\\
\sum_{l=1}^{4}\left(\alpha_{l}+\frac{\beta_{l}}{1-\overline{a_{0}}z}\right)\left(\overline{\alpha_{l}}+\frac{\overline{\beta_{l}}}{1-a_{0}\overline{w}}\right)\left(1-\psi_{\infty}(z)\psi_{\infty}(w)^{*}\right)\\
+\int_{C\backslash\{\infty\}}c(\psi)\left(\frac{1-\overline{a_{\psi}}z}{1-\overline{a_{0}}z}\right)\left(\frac{1-a_{\psi}\overline{w}}{1-a_{0}\overline{w}}\right)\left(1-\psi(z)\psi(w)^{*}\right)d\mu(\psi)\end{multline*}
 for some positive $c\in L^{1}(\mu)$ and some $\alpha_{1},\,\beta_{2},\,\ldots,\, a_{4},\,\beta_{4}\in\mathbb{C}$.
We know that the $\psi$s are Blaschke products, and we know their
roots, so we can write this even more explicitly as \begin{multline*}
1-z^{2}\overline{w}^{2}\frac{z-a_{0}}{1-\overline{a_{0}}z}\frac{\overline{w}-\overline{a_{0}}}{1-a_{0}\overline{w}}=\\
\sum_{l=1}^{4}\left(\alpha_{l}+\frac{\beta_{l}}{1-\overline{a_{0}}z}\right)\left(\overline{\alpha_{l}}+\frac{\overline{\beta_{l}}}{1-a_{0}\overline{w}}\right)\left(1-z^{2}\overline{w}^{2}\right)\\
+\int_{C\backslash\{\infty\}}c(\psi)\frac{1-\overline{a_{\psi}}z}{1-\overline{a_{0}}z}\frac{1-a_{\psi}\overline{w}}{1-a_{0}\overline{w}}\left(1-z^{2}\overline{w}^{2}\frac{z-a_{\psi}}{1-\overline{a_{\psi}}z}\frac{\overline{w}-\overline{a_{\psi}}}{1-a_{\psi}\overline{w}}\right)d\mu(\psi)\end{multline*}
 If we multiply both sides by $\left(1-\overline{a_{0}}z\right)\left(1-a_{0}\overline{w}\right)$
we get \begin{multline*}
\left(1-\overline{a_{0}}z\right)\left(1-a_{0}\overline{w}\right)-z^{2}\overline{w}^{2}\left(z-a_{0}\right)\left(\overline{w}-\overline{a_{0}}\right)=\\
\sum_{l=1}^{4}\left(\alpha_{l}\left(1-\overline{a_{0}}z\right)+\beta_{l}\right)\left(\overline{\alpha_{l}}\left(1-a_{0}\overline{w}\right)+\overline{\beta_{l}}\right)\left(1-z^{2}\overline{w}^{2}\right)\\
+\int_{C\backslash\{\infty\}}c(\psi)\left(\left(1-\overline{a_{\psi}}z\right)\left(1-a_{\psi}\overline{w}\right)-z^{2}\overline{w}^{2}\left(z-a_{\psi}\right)\left(\overline{w}-\overline{a_{\psi}}\right)\right)d\mu(\psi)\end{multline*}
 which we can expand to get \begin{multline*}
1-\overline{a_{0}}z-a_{0}\overline{w}+\left|a_{0}\right|^{2}z\overline{w}-z^{3}\overline{w}^{3}+a_{0}z^{2}\overline{w}^{3}+\overline{a_{0}}z^{3}\overline{w}^{2}-\left|a_{0}\right|^{2}z^{2}\overline{w}^{2}=\\
\sum_{l=1}^{4}\left\{ \begin{array}{c}
\left|\alpha_{l}+\beta_{l}\right|^{2}-\left|\alpha_{l}+\beta_{l}\right|^{2}z^{2}\overline{w}^{2}\\
-(\alpha_{l}+\beta_{l})\overline{\alpha_{l}}a_{0}\overline{w}+(\alpha_{l}+\beta_{l})\overline{\alpha_{l}}a_{0}z^{2}\overline{w}^{3}\\
-(\overline{\alpha_{l}}+\overline{\beta_{l}})\alpha_{l}\overline{a_{0}}z+(\overline{\alpha_{l}}+\overline{\beta_{l}})\alpha_{l}\overline{a_{0}}z^{3}\overline{w}^{2}\\
+\left|\alpha_{l}\right|^{2}\left|a_{0}\right|^{2}z\overline{w}-\left|\alpha_{l}\right|^{2}\left|a_{0}\right|^{2}z^{3}\overline{w}^{3}\end{array}\right.\\
+\int_{C\backslash\{\infty\}}c(\psi)\left(\begin{array}{c}
1-\overline{a_{\psi}}z-a_{\psi}\overline{w}+\left|a_{\psi}\right|^{2}z\overline{w}\\
-z^{3}\overline{w}^{3}+\overline{a_{\psi}}z^{3}\overline{w}^{2}+a_{\psi}z^{2}\overline{w}^{3}-\left|a_{\psi}\right|^{2}z^{2}\overline{w}^{2}\end{array}\right)d\mu(\psi)\end{multline*}

To get our contradiction, we look at the $\overline{w}$, $z^{3}\overline{w}^{3}$
and $z^{2}\overline{w}^{2}$ coefficients of this equation, i.e, \begin{align}
a_{0} & =\sum_{l=1}^{4}(\alpha_{l}+\beta_{l})\overline{\alpha_{l}}a_{0}+\int_{C\backslash\{\infty\}}c(\psi)a_{\psi}d\mu(\psi)\label{eq:con-1}\\
1 & =\sum_{l=1}^{4}\left|\alpha_{l}\right|^{2}\left|a_{0}\right|^{2}+\int_{C\backslash\{\infty\}}c(\psi)d\mu(\psi)\label{eq:con-2}\\
\left|a_{0}\right|^{2} & =\sum_{l=1}^{4}\left|\alpha_{l}+\beta_{l}\right|^{2}+\int_{C\backslash\{\infty\}}c(\psi)\left|a_{\psi}\right|^{2}\,.\label{eq:con-3}\end{align}
 We can easily see that \eqref{eq:con-1} implies \[
\left|a_{0}\right|^{2}=\left|\sum_{l=1}^{4}(\alpha_{l}+\beta_{l})\overline{\alpha_{l}}a_{0}+\int_{C\backslash\{\infty\}}c(\psi)a_{\psi}d\mu(\psi)\right|^{2}\]
 and if we define a Hilbert space $H=\mathbb{C}^{4}\oplus L^{2}(\mu)$,
then we can rewrite this as\begin{equation}
\left|a_{0}\right|^{2}=\left|\left\langle \left(\begin{matrix}(a_{0}\overline{\alpha_{l}})_{l}\\
c(\psi)^{\nicefrac{1}{2}}\end{matrix}\right),\,\left(\begin{matrix}(\overline{\alpha_{l}}+\overline{\beta_{l}})_{l}\\
c(\psi)^{\nicefrac{1}{2}}\overline{a_{\psi}}\end{matrix}\right)\right\rangle \right|^{2}\label{eq:vectors}\end{equation}
 and apply the Cauchy-Schwarz inequality, so \begin{multline*}
\left|a_{0}\right|^{2}\leq\\
\underbrace{\left(\sum_{l=1}^{4}\left|\alpha_{l}\right|^{2}\left|a_{0}\right|^{2}+\int_{C\backslash\{\infty\}}c(\psi)d\mu(\psi)\right)}_{=1\text{ by \eqref{eq:con-2}}}\underbrace{\left(\sum_{l=1}^{4}\left|\alpha_{l}+\beta_{l}\right|^{2}+\int_{C\backslash\{\infty\}}c(\psi)\left|a_{\psi}\right|^{2}\right)}_{=\left|a_{0}\right|^{2}\text{ by \eqref{eq:con-3}}}\\
=\left|a_{0}\right|^{2}\,.\end{multline*}

Further, the two vectors in \eqref{eq:vectors} are linearly independent,
as $\overline{a_{\psi}}$ is {[}the complex conjugate of] the third
root of $\psi$, which is different for each $\psi$, so the inequality
is strict, which is a contradiction%
\footnote{This is a slight oversimplification. If $c(\psi)$ is non-zero at
exactly one point $\psi$ (and that $\psi$ is singular with respect
to $\mu$), then this argument doesn't hold. However, a simple calculation
by equating coefficients (which is omitted) shows that if the vectors
are linearly dependent, so $c(\psi)$ is non-zero at exactly one point,
then that point must be $\psi_{0}\notin C$, which is also a contradiction.%
}.
\end{proof}
\begin{rem}
Clearly, this proof is not very good. It would seem more natural to
use the fact that the Herglotz-Agler representation from Theorem \ref{thm:SuperHerglotz}
was parameterised by extreme measures, although there is no obvious
way to do this.
\end{rem}
The fact that the minimum set of test functions is parameterised by
the sphere is interesting, as the set of kernels given in \cite{H1Infinity}
is also parameterised by the sphere, and conjectured to be minimal
(that paper contains some partial results, towards this aim).

Similarly, Abrahamse gave (in \cite{AbrahmsePick}) a set of kernels
corresponding to interpolation on a multiply connected ($n$-holed)
domain, parameterised by the $n$-torus, and conjectured that this
set of kernels was minimal (there are some partial results in this
direction in \cite{BallClancey}); in \cite{DritschelInterpolation}
and \cite{Hyperelliptic}, the authors give sets of test functions
for $n$-holed domains, which are also parameterised by the $n$-torus,
and minimal%
\footnote{The main aim of \cite{Hyperelliptic} was not to investigate test
functions, although Note 2.10 of \cite{Hyperelliptic} gives the set
of test functions, and the results of Section 5.2 of \cite{Hyperelliptic}
are broadly analogous to the results used in the proof of Theorem
\ref{thm:Minimal} in this paper, and in Proposition 5.3 of \cite{DritschelInterpolation}.%
}.

It's not clear whether there is some sort of duality between minimal
sets of test functions and minimal sets of kernels. A possible counterexample
to such a duality is the bidisc; it's well known (see for example
\cite{AglerPick}) that only two test functions are needed for the
bidisc, whereas in \cite{CStarEnvelopesMcCulloughPaulsen}, the authors
conjecture that infinitely many kernels are required.

\subsection{Generalisation}

It's worth briefly discussing how these ideas could be generalised
to other spaces. In \cite{RaghupathiCBHInfty}, the author looks at
spaces of the form%
\footnote{Here, $B$ is a Blaschke product%
} $\mathbb{C}+BH^{\infty}$. These spaces are a natural generalisation
of the space $H_{1}^{\infty}$, and the author provides a generalisation
of the Nevanlinna-Pick theorem from \cite{H1Infinity}%
\footnote{The author only generalises the first form here; a generalisation
of the second form is given in \cite{MatrixH1Infinity}.%
}.

We used the Herglotz representation trick to turn a linear equation
($f^{\prime}(0)=0$) into a constraint on probability measures (that
they have zero mean), found the extreme points of the set of constrained
probability measures, and then used those extreme probability measures
to generate our test functions.

Suppose we want to apply these techniques to $\mathbb{C}+BH^{\infty}$.
It's fairly clear that we can come up with a set of linear equations
that all functions in $\mathbb{C}+BH^{\infty}$ must satisfy (functions
must be constant at zeroes of $B$, and have a prescribed number of
zero derivatives at repeated zeroes of $B$). It should also be fairly
easy to turn this set of linear equations into a set of constraints
on probability measures.

The difficulty comes when we calculate extreme points. Theorem \ref{thm:ExtremeDirections}
should work well enough, so if we have $n$ independent equations,
we can say that extreme measures are supported on at most $2n+1$
points. However, the reasoning that followed Theorem \ref{thm:ExtremeDirections}
(which showed precisely which such measures would give functions in
$H_{1}^{\infty}$) relied on a geometric interpretation of the probability
constraint, which doesn't obviously generalise to other types of constraint.

If we can calculate the extreme measures, these should generate test
functions in precisely the same way as above. However, the proof that
the test functions we've found are minimal is heavily dependent --
perhaps overly dependent -- on explicit calculations using the test
functions.

\bibliographystyle{amsalpha}
\bibliography{/home/james/Documents/Maths/library/library}

\providecommand{\bysame}{\leavevmode\hbox to3em{\hrulefill}\thinspace}
\providecommand{\MR}{\relax\ifhmode\unskip\space\fi MR }
\providecommand{\MRhref}[2]{%
  \href{http://www.ams.org/mathscinet-getitem?mr=#1}{#2}
}
\providecommand{\href}[2]{#2}
\begin{thebibliography}{DPRS07}

\bibitem[Abr79]{AbrahmsePick}
Marine~B. Abrahamse, \emph{The {P}ick interpolation theorem for finitely
  connected domains}, Michigan Math. J. \textbf{26} (1979), no.~2, 195--203.

\bibitem[AHR07]{AglerComputation}
Jim Agler, John Harland, and Benjamin Raphael, \emph{Classical function theory,
  operator dilation theory, and machine computations on multiply connected
  domains}, Memoirs of the American Mathematical Society, AMS Bookstore, 2007.

\bibitem[AM02]{AglerPick}
Jim Agler and John~E. McCarthy, \emph{Pick interpolation and {H}ilbert function
  spaces}, Graduate Studies in Mathematics, AMS, 2002.

\bibitem[BBT08]{MatrixH1Infinity}
Joseph~A. Ball, Vladimir Bolotnikov, and Sanne {Ter Horst}, \emph{A constrained
  {N}evanlinna-{P}ick interpolation problem for matrix-valued functions},
  ArXiv: 0809:2345, September 2008.

\bibitem[BC96]{BallClancey}
Joseph~A. Ball and Kevin~F. Clancey, \emph{Reproducing kernels for {H}ardy
  spaces on multiply connected domains}, Integral Equations and Operator Theory
  \textbf{25} (1996), 35--57.

\bibitem[DM07]{DritschelInterpolation}
Michael~A. Dritschel and Scott McCullough, \emph{Test functions, kernels,
  realizations and interpolation}, Operator Theory, Structured Matrices and
  Dilations: Tiberiu Constantinescu Memorial Volume, pp.~153--179, Theta
  Foundation, Bucharest, 2007.

\bibitem[Dou66]{DouglasLemma}
Ronald~G. Douglas, \emph{On majorization, factorization, and range inclusion of
  operators on {H}ilbert space}, Proceedings of the American Mathematical
  Society \textbf{17} (1966), no.~2, 413--415.

\bibitem[DPRS07]{H1Infinity}
Kenneth~R. Davidson, Vern Paulsen, Mrinal Raghupathi, and Dinesh Singh, \emph{A
  constrained {N}evanlinna-{P}ick interpolation problem}, ArXiv:0711.2032, Nov
  2007.

\bibitem[MP02]{CStarEnvelopesMcCulloughPaulsen}
Scott McCullough and Vern Paulsen, \emph{{$C^*$}-envelopes and interpolation
  theory}, Indiana University Mathematics Journal \textbf{51} (2002), no.~2,
  479--505.

\bibitem[Pic08]{Hyperelliptic}
James Pickering, \emph{Counterexamples to rational dilation on symmetric
  multiply connected domains}, Complex Analysis and Operator Theory (2008),
  Online edition, yet to appear in print.

\bibitem[Rag08]{RaghupathiCBHInfty}
Mrinal Raghupathi, \emph{Nevanlinna-{P}ick interpolation for
  {$\mathbb{C}+BH^\infty$}}, ArXiv:0803.1278, 2008.

\end{thebibliography}

\end{document}